\documentclass[11pt,english,leqno]{amsart}
\title{Resolutions for principal series representations  of $p$-adic ${\rm GL}_n$}

\author{Rachel Ollivier}

\address{Columbia University, Mathematics, 2990 Broadway, New York, NY 10027}
\email{ollivier@math.columbia.edu}
\dedicatory{To Peter Schneider, on the occasion of his 60th birthday.}
 
\usepackage[latin1]{inputenc}
\usepackage{graphicx}

\usepackage[all,color,matrix,arrow,curve]{xypic}
\usepackage{frcursive}
\makeatletter
\def\@tocline#1#2#3#4#5#6#7{\relax
  \ifnum #1>\c@tocdepth 
  \else
    \par \addpenalty\@secpenalty\addvspace{#2}%
    \begingroup \hyphenpenalty\@M
    \@ifempty{#4}{%
      \@tempdima\csname r@tocindent\number#1\endcsname\relax
    }{%
      \@tempdima#4\relax
    }%
    \parindent\z@ \leftskip#3\relax \advance\leftskip\@tempdima\relax
    \rightskip\@pnumwidth plus4em \parfillskip-\@pnumwidth
    #5\leavevmode\hskip-\@tempdima
      \ifcase #1
       \or\or \hskip 1em \or \hskip 2em \else \hskip 3em \fi%
      #6\nobreak\relax
    \dotfill\hbox to\@pnumwidth{\@tocpagenum{#7}}\par
    \nobreak
    \endgroup
  \fi}
\makeatother

\usepackage{hyperref, endnotes}

\usepackage{mathcmd}
\usepackage{amssymb,amsmath, amscd,upgreek, amsthm}
\usepackage{mathrsfs}
\usepackage{euscript}

\usepackage[usenames,dvipsnames]{xcolor}

\numberwithin{equation}{section}
\theoremstyle{plain}{}
\newtheorem{theorem}{Theorem}[section]

\newtheorem{fact}{Fact}

\newtheorem{coro}[theorem]{Corollary}
\theoremstyle{plain}
\newtheorem{prop}[theorem]{Proposition}

\newtheorem{lemma}[theorem]{Lemma}

\theoremstyle{remark}



\def\Corps{\mathfrak F}

\def\F{{F}}
\def\G{{\rm G}}
\def\I{{\rm I}}

\def\KK{{\rm  K}}
\def\K{\KK}

\def\N{{\mathbb N}}

\def\R{{\rm R}}

\def\T{{\rm T}}

\def\V{{\rm V}}
\def\W{{\rm W}}
\def\X{{\rm X}}

\def\Z{{\mathbb Z}}

\def\Ap{\mathscr{A}}

\def\Aa{\mathcal{A}}
\def\B{{\rm B}}
\def\U{{\rm U}}

\def\Dd{\EuScript{D}}

\def\Ff{\mathscr{F}}
\def\Hh{{ {\rm H}}}

\def\Oo{\mathfrak{O}}

\def\Pp{\EuScript{P}}

\def\k{{\mathbf k}}

\def\Iw{{\rm  I'}}

\def\Tp{{\rm T}}

\def\Gp{ { \rm G}}
\def\pr{{\rm  pr}}
\def\char{{\rm char}}

\newcommand{\cX}{\underline{\underline{\mathbf{X}}}}
\newcommand{\cV}{\underline{\underline{\mathbf{V}}}}

\def\lp{\langle}
\def\rp{\rangle}

\def\ind{{\rm ind}}
\def\Ind{{\rm Ind}}

\def\1{{\mathbf 1}}

\def\XX{{\mathbf X}}

\def\H{ {{\mathfrak H}}}

\def\Wf{\mathfrak W}

\def\val{\operatorname{\emph{val}}}

\usepackage{comment}

 \theoremstyle{remark}%
\usepackage{hyperref,url}
\definecolor{webblue}{rgb}{0, 0.7, 0.5}
\definecolor{webred}{rgb}{0.2, 0.3, 0.6} 
\hypersetup{colorlinks,  citecolor=webblue,  linkcolor=webred, urlcolor=black}

\title{Resolutions for principal series representations  of $p$-adic ${\rm GL}_n$}

\dedicatory{\hspace{6cm}To Peter Schneider on the occasion of his birthday.\\}
\usepackage{csquotes}

\urladdr{\url{http://www.math.columbia.edu/~ollivier}}

\begin{document}

\maketitle

\begin{abstract}  Let $\mathfrak F$ be a nonarchimedean locally compact field with residue characteristic $p$ and $\mathbf G(\mathfrak F)$ the group of $\mathfrak F$-rational points of a connected reductive group.  In \cite{SS}, Schneider and Stuhler realize, in a functorial way, any
 smooth  complex  finitely generated representation of $\mathbf G(\mathfrak F)$ as the $0$-homology of a certain coefficient system  on the semi-simple building of $\mathbf G$.
It is known that this method does not apply  in general for  smooth mod $p$ representations of $\mathbf G(\mathfrak F)$,   even when $\mathbf G= {\rm GL}_2$. However, we prove that
a principal series representation of ${\rm GL}_n(\mathfrak F)$ over a field with arbitrary characteristic can be realized as the $0$-homology of  the corresponding coefficient system as defined in \cite{SS}.

\end{abstract}

\setcounter{tocdepth}{1}

\tableofcontents
\section{Introduction} Let $\mathfrak F$ be a nonarchimedean locally compact field with residue characteristic $p$ and $\mathbf G(\mathfrak F)$ the group of $\mathfrak F$-rational points of a connected reductive group $\mathbf G$.
By a result of Bernstein, the blocks of the category of smooth complex representations of  $\mathbf G(\mathfrak F)$ have  finite global dimension. The $\mathbf G(\mathfrak F)$-equivariant coefficient systems on the semisimple
building  $\mathscr X$ of $\mathbf G$ introduced in  \cite{SS} allow Schneider and Stuhler  to construct explicit projective resolutions  for finitely generated representations in this category.  One of the  key ingredients for their result is the following fact, which is valid over an arbitrary field $\k$:
consider the (level $0$) universal representation $\XX=\k[\I\backslash \mathbf G(\mathfrak F)]$ where $\I$ is a fixed pro-$p$ Iwahori subgroup of $\mathbf G(\mathfrak F)$, then the
attached coefficient system  $\cX$ on $\mathscr X$ gives the following  \emph{exact} augmented chain complex
\begin{equation}\label{chain-complexintro}
    0 \longrightarrow C_c^{or} (\mathscr{X}_{(d)}, \cX) \xrightarrow{\;\partial\;} \ldots \xrightarrow{\;\partial\;} C_c^{or} (\mathscr{X}_{(0)}, \cX) \xrightarrow{\;\epsilon\;} \mathbf{X} \longrightarrow 0
\end{equation}
of $\mathbf G(\mathfrak F)$-representations and of left $\Hh:=\k[\I\backslash \mathbf G(\mathfrak F)/\I]$-modules (see \S\ref{therez} below for the notation).

\medskip

If $\k$ has characteristic  $p$, it is no longer true  that the category of smooth representations  of $\mathbf G(\mathfrak F)$ generated by their pro-$p$ Iwahori fixed vectors has finite global dimension:  in the case of ${\rm PGL }_2(\mathbb Q_p)$,
this category is equivalent to the category of modules over  $\Hh$ (\cite{Inv})
and it is proved in \cite{OS} that the latter has infinite global dimension.

Still if $\k$ has characteristic  $p$, it is  also no longer true   that any $\mathbf G(\mathfrak F)$-representation $\mathbf V$ generated by its $\I$-invariant subspace 
can be realized as the $0$-homology of the  coefficient system $\cV$ defined  as in \cite{SS}: it is true for the universal representation $\mathbf{X}$ as noted above, but  \cite[Remark 3.2]{OS} points out  a counter-example when $\mathbf V$ is  a supercuspidal representation of ${\rm GL }_2(\mathbb Q_p)$.
However, 
 realizing  any smooth irreducible $\k$-representation of ${\rm GL }_2(\mathbb Q_p)$ 
as  the $0$-homology of some  finite dimensional coefficient system is important  in Colmez's construction  of  a functor yielding the $p$-adic  local Langlands correspondence     (\cite{Co}).  
As explained in \cite{OSe},  the resolutions  in \cite{Co} can be retrieved in the following   way:  let $\mathbf V$  be a smooth representation  of ${\rm GL }_2(\mathbb Q_p)$  with a central character and generated by its $\I$-invariant subspace  $\mathbf V^\I$, then  by the equivalence of categories of \cite{Inv}, tensoring \eqref{chain-complexintro} by the  $\Hh$-module $\mathbf V^\I$ gives an exact resolution of $\mathbf V\simeq \mathbf V^{\I}\otimes_\Hh \XX$. 
But the  equivalence of categories does not hold in general. For arbitrary $\mathfrak F$, Hu attaches    to any irreducible representation of  ${\rm GL }_2(\mathfrak F)$ with central character a coefficient system on the tree whose $0$-homology is isomorphic to $\mathbf V$ (\cite{Hudiag}).  This coefficient system, although not finite dimensional in general, turns out to be finite dimensional  when $\mathfrak F=\mathbb Q_p$  and one retrieves, once again, the resolutions of \cite{Co}. 
But if $\mathfrak F$ has positive characteristic (respectively  if $\mathfrak F/\mathbb Q_p$ is a  quadratic unramified extension), then  for $\mathbf V$ supercuspidal,  there  is no finite dimensional
 coefficient system whose $0$-homology is isomorphic to $\mathbf V$ as proved in  \cite{Hudiag}     (respectively   in \cite{Schraen}).

\medskip
Most of the surprising phenomena occurring in the smooth mod $p$ representation theory of $\mathbf G(\mathfrak F)$
are related to the properties of the supercuspidal representations, whereas the behavior of the principal series representations is easier to analyze and is somewhat similar to what is observed in the setting of complex representations. 
To formalize this remark, Peter Schneider asked me the following question: the Hecke algebra  $\Hh$  contains  a copy $\Aa_{anti}$ of the $\k$-algebra of the semigroup of all  (extended) antidominant cocharacters of a split torus of $\mathbf G$; is $\Hh$ 
 free as a $\Aa_{anti}$-module when localized at a regular character of $\Aa_{anti}$? 
 (see  \S\ref{lesAlg} and \S\ref{sec:over-a-ring}  for the definitions  and Propositions \ref{prop:iso} and \ref{prop:invariants} for the link between regular characters of $\Aa_{anti}$ and principal series representations). The answer  is yes    and  the present note is largely inspired by this 
question. We prove the following theorem, where $n$ is an integer $\geq 1$.

\begin{theorem}
\label{theo}
Let  $\k$ be an arbitrary field  and $\mathbf V$ a smooth principal series representation of ${\rm GL}_n(\mathfrak F)$ over $\k$. Let $\cV$ 
be the coefficient system associated to $\mathbf V$ as in \cite{SS}.  Then 
the following augmented chain complex
\begin{equation}\label{chain-complexintroV}
    0 \longrightarrow C_c^{or} (\mathscr{X}_{(n-1)}, \cV) \xrightarrow{\;\partial\;} \ldots \xrightarrow{\;\partial\;} C_c^{or} (\mathscr{X}_{(0)}, \cV) \xrightarrow{\;\epsilon\;} \mathbf{V} \longrightarrow 0
\end{equation} yields an exact resolution of $\mathbf V$ as a representation of ${\rm GL}_n(\mathfrak F)$.

\end{theorem}

This theorem is proved in Section \ref{sec:reso}. In the previous sections, the arguments are written in the setting of a general split group. However, in Section \ref{sec:reso}, we need an extra geometric condition on the facets of the standard apartment to be able to fully use Iwasawa decomposition. Therefore, we restrict ourselves to the case of ${\rm GL}_n$. We suspect that Theorem
\ref{theo} is true in general.  

\medskip

A generalization of Colmez' functor  to reductive groups  over $\mathbb Q_p$
    is proposed by Schneider and Vign\' eras in \cite{SV}. The first fundamental construction is the one of   a universal $\delta$-functor $V\mapsto D^i( V)$ for $i\geq 0$, from the category $\mathcal M_{o-tor}(\B)$ to the category $\mathcal M_{{e}t}(\Lambda(\B^+))$. The ring $o$ is the ring of integers of a  fixed finite extension of $\mathbb Q_p$ and $\mathcal M_{o-tor}(\B)$ is the abelian category of smooth representations of $\B$ in $o$-torsion modules, where $\B$ is a fixed Borel subgroup in $\mathbf G(\mathbb Q_p)$. We do not describe the category  $\mathcal M_{{e}t}(\Lambda(\B^+))$ explicitly here. 
The restriction $V$ to $\B$ of a principal series representation of $\mathbf G(\mathbb Q_p)$ over a field with characteristic $p$ is an example of  an object in 
$\mathcal M_{o-tor}(\B)$ for a suitable ring $o$.  
Using  Theorem \ref{theo}, we prove the following result (\S\ref{sec:sv}):

\begin{prop} Suppose that $\k$ is a field with characteristic $p$. Let $V$ be the restriction to $\B$ of a  principal series  representation of ${\rm GL}_n(\mathbb Q_p)$ over $\k$. Then $D^i(V)=0$ for $i\geq n-1$.

\end{prop}

\section{Notations and preliminaries}
\subsection{} From now on we suppose that 
$\mathbf{G}$ is $\mathfrak{F}$-split and we denote ${\mathbf  G}(\mathfrak{F})$ by $\G$. 
We  fix a uniformizer $\varpi$  for $\mathfrak F$ and choose the valuation $\val_{\mathfrak{F}}$ on   $\mathfrak{F}$  normalized by $\val_{\mathfrak{F}}(\varpi)=1$. The  ring of integers of $\Corps$ is denoted by $\Oo$ and its residue field   with cardinality $q= p^f$ by $\mathbb F_q$.

 In  the semisimple  building $\mathscr {X}$ of $\Gp$, we choose the chamber $C$
 corresponding to the Iwahori subgroup $\Iw$ that contains the  pro-$p$  subgroup  $\I$.  This choice is unique up to conjugacy by an element of $\Gp$. 
Since $\G$ is split, $C$ has at least one  hyperspecial vertex    $x_0$ and we denote by $\K$ the  associated maximal compact subgroup of $\Gp$.
We fix a maximal $\Corps$-split torus $\T$ in $\Gp$ such that
the corresponding  apartment $\Ap$ in $\mathscr{X}$ contains $C$.

Let $\mathbf{G}_{x_0}$ and $\mathbf{G}_C$ denote the Bruhat-Tits group schemes over $\mathfrak{O}$ whose $\mathfrak{O}$-valued points are $\K$ and $\Iw$ respectively. 
   Their reductions over the residue field $\mathbb{F}_q$ are denoted by $\overline{\mathbf{G}}_{x_0}$ and $\overline{\mathbf{G}}_C$. 
   Note that $\Gp= \mathbf{G}_{x_0}(\Corps)=\mathbf{G}_{C}(\Corps)$.
   By \cite[3.4.2, 3.7 and 3.8]{Tit}, 
$\overline{\mathbf{G}}_{x_0}$ is connected  reductive and $\mathbb F_q$-split.
Therefore we have ${\mathbf{G}}_C^\circ(\Oo)={\mathbf{G}}_C(\Oo) = \Iw$ and ${\mathbf{G}}_{x_0}^\circ(\Oo)={\mathbf{G}}_{x_0}(\Oo) = \K$.   Denote    by $\K_1$ the pro-unipotent radical of $\K$.   
More generally we consider the fundamental system of open normal subgroups
\begin{equation*}
    \K_m := \ker \big(\mathbf{G}_{x_0}(\Oo) \xrightarrow{\; \pr \;} \mathbf{G}_{x_0} (\Oo/\varpi^m \Oo) \big) \qquad\textrm{for $m \geq 1$}
\end{equation*}
of $\K$.
The quotient $\K/\K_1$ is isomorphic to  $\overline{\mathbf G}_{x_0}(\mathbb F_q)$. 
The Iwahori subgroup $\Iw$  is the preimage in $\K$ of the $\mathbb F_q$-rational points of  a Borel subgroup $\overline{\mathbf B}$
with Levi decomposition $\overline{\mathbf B}= \overline{\mathbf T}\,\overline{\mathbf N}$. 
The pro-$p$ Iwahori subgroup $\I$ is the preimage in $\Iw$ of $\overline{\mathbf N}(\mathbb F_q)$. The preimage of  $ \overline{\mathbf{T}}(\mathbb F_q)$ is the the maximal compact subgroup $\Tp^0$ of $\Tp$.
Note that $\T^0/\T^1=\Iw/\I =  \overline{\mathbf{T}}(\mathbb F_q)$ where  $\T^1:= \T^0\cap\I$.

\medskip

 To the choice of $\T$ is attached  the root datum $(\Phi, {\rm X}^*(\Tp), \check\Phi, {\rm X}_*(\Tp))$.
This root system is reduced because  the group $\mathbf{G}$ is $\mathfrak F$-split. We denote by $\Wf$ the finite Weyl group  $N_\G(\T)/\T$, quotient by $\T$ of the normalizer of $\T$.   Let $\lp\,. \,,\, .\,\rp$ denote the perfect pairing ${\rm X}_*(\Tp)\times {\rm X}^*(\Tp)\rightarrow \Z$.
The elements in $\X_*(\T)$ are the cocharacters of $\T$ and we will call them the  coweights.
We identify the set  $\X_*(\T)$ with the subgroup  $\T/\T^0$ of the extended Weyl group $\W= N_\G(\T)/\T^0$ as in \cite[I.1]{Tit} and \cite[I.1]{SS}: to an element $g\in \Tp$ corresponds a vector $\nu(g)\in \mathbb R\otimes _{\mathbb Z}\X_*({\Tp})$ defined by
\begin{equation}\label{normalization}
    \lp\nu(g),\, \chi\rp =-\val_{\mathfrak F}(\chi(g))  \qquad \textrm{for any } \chi\in \X^*(\Tp).
\end{equation} and $\nu$ induces the required isomorphism   $\Tp/\T^0\cong\X_*(\Tp)$. Recall that  $\Ap$ denotes the apartment of the semisimple building attached to  $\Tp$ (\cite{Tit} and \cite[I.1]{SS}).  The group $\Tp/\T^0$ acts by translation on $\Ap$ via $\nu$.  
The actions of $\Wf$ and $\Tp/\T^0$ combine into an action of 
 $\W$  on $\Ap$ as recalled in \cite[page 102]{SS}. Since $x_0$ is a special vertex of the building,   $\W$  is isomorphic to the semidirect product $\X_*(\T)\rtimes \Wf$  where we see $\mathfrak W$ as the fixator in $\W$ of any point in the extended apartment lifting $x_0$ (\cite[1.9]{Tit}).
A coweight $\lambda$ will sometimes be denoted by $e^\lambda$ to underline that we see it as an element in $\W$, meaning as a translation on the semisimple apartment $\Ap$. 

\medskip

 We see the roots $\Phi$ as the set of affine roots taking value zero at $x_0$. Then $\Phi^+$ is  the set of roots in $\Phi$ taking non negative values on $C$.
  The set of dominant coweights  $\X_*^+(\Tp)$ is the set of all $\lambda\in \X_*(\Tp)$ such that $\lp \lambda, \alpha\rp\geq 0$ for all $\alpha\in \Phi^+$.  A  coweight is called antidominant if its opposite is dominant.
A coweight $\lambda$ such that  $\lp \lambda, \alpha\rp< 0$ for all $\alpha\in \Phi^+$ is called strongly antidominant.

 \medskip
 
\subsection{}We fix a lift $\hat w\in N_\G(\T)$ for any $w\in  \W$. By Bruhat decomposition,  $\Gp$ is the disjoint union of all $\Iw \hat w \Iw$ for $w\in\W$. Recall that $\Tp^1$ is the pro-$p$ Sylow subgroup of $\T^0$.
We denote by $\tilde\W$ the quotient of $N_\Gp(\Tp)$   by $\Tp^1$  and obtain the exact sequence
$$0\rightarrow \Tp^0/\Tp^1 \rightarrow \tilde\W\rightarrow \W\rightarrow 0.$$ 
The group $\tilde\W$  parametrizes the double cosets of $\Gp$ modulo $\I$.  We fix a lift $\hat w\in N_\G(\T)$ for any $w\in \tilde \W$.  For $Y$ a subset of $\W$, we denote by $\tilde Y$ its preimage in $\tilde \W$. In particular, we have   the preimage $\tilde \X_*(\T)$  of $\X_*(\T)$. As well as those of $\X_*(\T)$, its elements will be denoted by $\lambda$ or $e^\lambda$ and called coweights. 
Note that a system of representatives of $\T/\T^1$ is given by the set of all $\widehat {e^\lambda}$ for $\lambda\in \tilde \X_*(\T)$. In fact, we recall that the  map
\begin{equation}\label{splitting}\lambda\in \X_*(\T)\rightarrow [ \lambda(\varpi^{-1})\,{\rm mod} \,\T^1 ]\in \tilde \X_*(\T)\end{equation}
  is a $\Wf$-equivariant splitting for   the exact sequence of abelian groups
\begin{equation}\label{eq:split}0\longrightarrow \T^0/\T^1\longrightarrow \tilde\X_*(\T) \longrightarrow \X_*(\T)\longrightarrow 0.\end{equation}
We will identify $\X_*(\T)$ with its image in $\tilde \X_*(\T)$ \emph{via} \eqref{splitting}.

\medskip

For $\alpha\in \Phi$, we inflate the function $\lp \,.\, , \alpha\rp$ defined on $\X_*(\T)$ to 
$\tilde \X_*(\T)$.   We still call \emph{dominant coweights} (resp.  \emph{antidominant coweights}) the elements in the preimage  $\tilde \X_*^+(\T)$  (resp. $\tilde \X_*^-(\T)$) of $\X_*^+(\T)$ (resp. $\X_*^-(\T)$).

\medskip

The group $\tilde \W$ is equipped with a length function $\ell: \tilde \W\rightarrow \N$ that inflates the length function on $\W$ (\cite[Proposition 1]{Vig}).

\subsection{\label{lesAlg}} Let $\k$ be an arbitrary field.
We consider the pro-$p$ Iwahori-Hecke algebra $$\Hh=\k[\I\backslash \Gp/\I]$$
  of  $\k$-valued functions with compact support in $\I\backslash \Gp
/\I$ under convolution.
 For $w\in \tilde\W$,  denote by $\tau_w$ the characteristic function of the double coset $\I \hat w \I$.  The set of all $(\tau_w)_{w\in \tilde \W}$ is a $\k$-basis for $\Hh$.
For $g\in \Gp$, we will also use the notation $\tau_g$ for the characteristic function of the double coset $\I g \I$. 
 In $\Hh$ we have the following relation, for $w$, $w'$ in $ \tilde \W$ (\cite[Theorem 50]{Vig}):
\begin{equation}\tau_{w} \tau_{w'}=\tau_{ ww'} \textrm{ if $\ell(w)+ \ell(w')=\ell(ww')$. }\end{equation}

It implies in particular that in $\Hh$ we have, for $\lambda$ and $\lambda'$ in $\tilde \X_*(\T)$:
\begin{equation}\label{eq:add}\tau_{e^{\lambda}} \tau_{e^{\lambda'}}=\tau_{ e^{\lambda+\lambda'}} \textrm{ if $\lambda$ and $\lambda'$ are   both antidominant.}\end{equation}
We denote by $\Aa_{anti}$  the commutative sub-$\k$-algebra of  $\Hh$ with $\k$-basis the set of all $\{\tau_{e^\lambda}, \: \lambda\in \tilde \X^-_*(\T)\}$.

\subsection{} 
Let $\U$ be the unipotent subgroup  of $\G$ generated by all the root subgroups  $\U_\alpha$ for $\alpha\in \Phi^+$ and $\B$ the Borel subgroup with Levi decomposition 
 $\B= \T\U$. Recall that we have $\G=\B \K$  since $x_0$ is a special vertex. Furthermore,  $\B\cap \K=\I'\cap \B$.

Let $\U^-$ denote the opposite unipotent subgroup   of $\G$ generated by all the root subgroups $\U_\alpha$ for $-\alpha\in \Phi^+$.
The pro-$p$ Iwahori subgroup $\I$ has the following decomposition:
$$\I= \: \I^+\:\I^0\:\I^-\textrm{ where  $\I^+:=\I\cap {\rm U}$,  $\I^0:=\I\cap \T=\T^1$, $\I^-:=\I\cap {\rm U}^-$}.$$

An element  $t\in \T$  contracts $\I^+$ and dilates $\I^-$ if it satisfies the conditions (see  \cite[(6.5)]{BK}):
\begin{equation} \label{cd}t \:\I^+ t^{-1}\subseteq \I^+, \qquad t^{-1} \I^- t\subseteq \I^-. \end{equation} Denote by $\T^{++}$ the semigroup of such $t\in\T$.

\begin{lemma}  We have $\T^{++}=\coprod_{\lambda\in \tilde \X_*^-(\T)} \T^1 \widehat{e^{\lambda}}$.

\label{negative}
\end{lemma}

\begin{proof} 
Let $\lambda\in \tilde\X_*(\T)$. It is proved in  \cite[Lemma 5.20]{Compa}
that the element $\widehat {e^\lambda}$ satisfies \eqref{cd} if and only if $\lambda$ is antidominant.


\end{proof}

\subsection{\label{remaANT}} We consider the $\k$-basis  $(E(w))_{w\in \tilde \W}$ for $\Hh$ as introduced in \cite{Vig}. Recall that $E(e^\lambda)=\tau_{e^\lambda}$ for all $\lambda\in\tilde \X_*^-(\T)$.
For $w\in \tilde\W$, there is  
$\lambda_0\in \tilde \X_*(\T)$ and $w_0\in \tilde\Wf$  such that 
$w= e^{\lambda_0} w_0$. 
Let $\lambda\in \tilde \X_*^-(\T)$ such that $\lambda+\lambda_0\in \tilde \X_*^-(\T)$. We claim that
\begin{equation}\label{mult}\tau_{e^\lambda} \: E(w)= q^{(\ell(w)+\ell(e^\lambda)-\ell(e^\lambda w))/2} E(e^{\lambda_0+\lambda})\tau_{w_0}\in \Aa_{anti}\:\tau_{w_0}.\end{equation}
The proof of this equality   given in the case of ${\rm GL}_n$
 in \cite[Proposition 4.8]{ANT} works in the general case with no modification.

\subsection{\label{strongl}} Let $t\in \T$ such that  the double class $\I \, t\, \I$ 
corresponds to a strongly antidominant element in $\X_*^-(\T)$.
The following lemma proved in \cite[Proposition 8, p.78]{SS1} is valid in the case of a general split reductive group. 

\begin{lemma}\label{lemmaSS} An open compact  subset  of $\B\backslash \Gp $   decomposes into a  finite disjoint union  of subsets of the form
 $\B \I t^n k= \B \I^- t^n k$ for $n$ large enough,
 where $k$ ranges over a finite subset of $\K$.
\end{lemma}

\begin{lemma} A system of neighborhoods of the identity in $\U^-$ is given by the set of all $t^{-m} \I^- t^m$ for $m\in\N$.\label{lemmaSS'}
\end{lemma}
\begin{proof}  A system of neighborhoods of the identity in $\U^-$ is given by the set of all $\K_m\cap \U^-$
and one checks that $t^{-m} \I^- t^m\subseteq \K_{m+1}\cap \U^-$ for all $m\in\N$.

\end{proof}
\section{Resolution of the level $0$ universal representation  of $\mathbf G(\mathfrak F)$\label{therez}}

We gather  here results from \cite{SS} and  use the  notations of \cite{OS}.
 We recall (cf.\ \cite{SS} I.1-2 for a brief overview) that the semi-simple building $\mathscr{X}$ is (the topological realization of) a $\Gp$-equivariant polysimplicial complex of dimension equal to the semisimple rank $d$ of $\Gp$. The (open) polysimplices are called facets and the $d$-dimensional, resp.\ zero dimensional, facets chambers, resp.\ vertices.
 For $i\in\{0, ..., d\}$, we denote by $\mathscr X_{(i)}$ the set of oriented facets of dimension $i$.
  Associated with each facet $F$ is, in a $\Gp$-equivariant way, a smooth affine $\mathfrak{O}$-group scheme $\mathbf{G}_F$ whose general fiber is $\mathbf{G}$ and such that $\mathbf{G}_F(\mathfrak{O})$ is the pointwise stabilizer in $\Gp$ of the preimage of $F$ in the extended building of $\Gp$. Its neutral component is denoted by $\mathbf{G}_F^\circ$ so that the reduction $\overline{\mathbf{G}}_F^\circ$ over $\mathbb{F}_q$ is a connected smooth algebraic group. The subgroup $\mathbf{G}_F^\circ(\mathfrak{O})$ of $\Gp$ is compact open. Let
\begin{equation*}
    \I_F := \{ g \in \mathbf{G}_F^\circ(\mathfrak{O}) :( g \mod \varpi) \in\ \textrm{unipotent radical of $\overline{\mathbf{G}}_F^\circ$} \}.
\end{equation*}

The $\I_F$ are compact open pro-$p$ subgroups in $\Gp$ which satisfy $\I_C = \I$, $\I_{x_0} = \K_1$,
\begin{equation}\label{pr1}
    g\I_F g^{-1} = \I_{gF} \qquad\textrm{for any $g \in \Gp$},
\end{equation}
and
\begin{equation}\label{pr2}
    \I_{F'} \subseteq \I_F \qquad\textrm{whenever $F' \subseteq \overline{F}$}.
\end{equation}

For any smooth $\k$-representation $\mathbf V$ of $\Gp$,  the family $\{\mathbf{V}^{\I_F}\}_F$ of subspaces of $\I_F$-fixed vectors in $\mathbf{V}$ forms a $\Gp$-equivariant coefficient system  on $\mathscr{X}$  which we will denote by $\cV$ (\cite{SS} II.2). Let $\mathbf X$ be the space $\k[\I\backslash\G]$ of $\k$-valued  functions with finite support in $\I\backslash\G$. It is a natural  left $\Hh$-module.
Let $\cX$ be the associated coefficient system. The corresponding augmented oriented chain complex
\begin{equation}\label{chain-complex}
    0 \longrightarrow C_c^{or} (\mathscr{X}_{(d)}, \cX) \xrightarrow{\;\partial\;} \ldots \xrightarrow{\;\partial\;} C_c^{or} (\mathscr{X}_{(0)}, \cX) \xrightarrow{\;\epsilon\;} \mathbf{X} \longrightarrow 0
\end{equation}
is a complex of $\Gp$-representations and of left $\Hh$-modules.


\medskip

As noticed in \cite[Remark 3.2]{OS}, the following result is contained in the proof of  \cite[Theorem II.3.1]{SS}:

\begin{theorem}[\cite{SS} Thm.\ II.3.1]
The complex \eqref{chain-complex} is exact.

\end{theorem}

Let $F$ be a facet in $\overline C$.   
Extending functions on  ${\mathbf G}_F^\circ(\mathfrak O)$ by zero to $\Gp$ induces a  ${\mathbf G}_F^\circ(\mathfrak O)$-equivariant embedding
\begin{equation*}
    \mathbf{X}_F := \k[\I\backslash {\mathbf G}_F^\circ(\mathfrak O)] \hookrightarrow \mathbf{X}   \ 
\end{equation*}
and we  consider the subalgebra
\begin{equation*}
    \H_F :=\k[\I\backslash {\mathbf G}_F^\circ(\mathfrak O)/\I]
\end{equation*} of the functions in $\Hh$ with support  in ${\mathbf G}_F^\circ(\mathfrak O)$.

\begin{lemma} The natural maps of  respectively
$(\mathbf{G}_F^\circ(\Oo),\H_{x_0}^{opp})$-bimodules and 
 $(\mathbf{G}_F^\circ(\Oo),\Hh^{opp})$-bimodules
\begin{equation}\label{transfF}\H_{x_0}\otimes _{\H_{F}}\XX_F\rightarrow \XX_{x_0}^{\I_F}\end{equation}
\begin{equation}\label{transf}\Hh\otimes _{\H_{F}}\XX_F\rightarrow \XX^{\I_F}\end{equation} 
are bijective.
\end{lemma}

\begin{proof}  The isomorphism \eqref{transf} is proved in \cite[Proposition 4.25]{OS}. The proof of  the bijectivity of \eqref{transfF},
is obtained similarly as follows. Let $\Phi^+_F$ denote the set of
positive roots that take value zero on $F$ and $\Dd_F$ the subset of all elements $d$ in $\Wf$ such that $d \Phi_F^+\subseteq \Phi^+$. Choose a lift $\tilde d\in \tilde \Wf$ for each such $d$. Then it is classical to establish that $\H_{x_0}$ is a free right  $\H_F$-module with basis $\{\tau_{\tilde d}\}_{d\in \Dd_F}$. Since $\H_F$ is Frobenius (\cite[Thm. 2.4]{Saw} and \cite[Prop 3.7]{Tin}), it
is self-injective: this implies that  the $\H_F$-module
$\H_{x_0}$ is  a direct summand of  $\Hh$ and the composition $\H_{x_0}\otimes _{\H_{F}}\XX_F \rightarrow  \Hh\otimes _{\H_{F}}\XX_F$ is an injective map inducing  the natural injection
\begin{equation*}\H_{x_0}\otimes _{\H_{F}}\XX_F\rightarrow \XX_{x_0}^{\I_F}.\end{equation*} 
To prove that it is surjective, we argue (again as in  \cite[Proposition 4.25]{OS}) using the fact that the set of all $\tilde d$ for $d\in \Dd_F$ yields a system of representatives for the double cosets
$\I\backslash \K/\mathbf{G}_F^\circ(\Oo)$ and that $\I \tilde d \I_F=\I \tilde d\I$.

\end{proof}

\medskip

We define $\EuScript P_F^\dagger$
to be the stabilizer of $F$ in $\Gp$. 
For $g \in \Pp_F^\dagger$, set  $\epsilon_F(g) = +1$, respectively $-1$, if $g$ preserves, respectively reverses, a given orientation of $F$. 
For any representation $\V$ of  $\Pp_F^\dagger$, we denote by $\V\otimes \epsilon_F$ the space $\V$ endowed with the structure of a  representation of $\Pp_F^\dagger$ where the action of $\Pp_F^\dagger$ is twisted by the character $\epsilon_F$.




\medskip

For  $i\in\{0, ..., d\}$, we fix a (finite) set of representatives $\mathscr{F}_i$ for the $\Gp$-orbits in $\mathscr{X}_i$ such that every member in $\mathscr{F}_i$ is contained in $\overline{C}$.
As explained in \cite[3.3.2]{OS}:
\begin{prop} \label{prop:sum}Let $i\in\{0, ..., d\}$. \begin{itemize}
\item[i.] The   $(\Gp, \Hh^{opp})$-bimodule $C_c^{or} (\mathscr{X}_{(i)}, \cX) $ is isomorphic to the direct sum
 $$ \bigoplus_{F\in \Ff_i}\ind_{\Pp_F^\dagger}^\Gp(   \mathbf{X} ^{\I_F}\otimes \epsilon_F).$$
\item[ii.]  In particular, as a left $\Aa_{anti}$-module, it is isomorphic to a direct sum of modules of the form $\XX^{\I_F}$ for $F\in \Ff_i$.
\end{itemize}
\end{prop}

\section{\label{sec:over-a-ring}Principal series representations over a ring}

Let $\R$ be a  commutative $\k$-algebra. Given a topological group $H$, we  consider $\R$-representations of $H$  that is to say $\R$-modules endowed with a $\R$-linear action of $H$. If the stabilizers 
of the points are open in $H$, then such a representation is called smooth. 
Let $\R^\times$  be the group of invertible elements in $\R$. A  morphism of $\k$-algebras
$\Aa_{anti}\rightarrow \R$ is called a character. If  the image of every element $\tau_{e^\lambda}$, $\lambda\in\tilde \X_*^-(\T)$ lies in $\R^\times$, then the character  is called   regular.

\medskip
\begin{lemma} \label{lemma:xi}There is a bijection $\phi\mapsto \overline \phi$ from the  set of morphisms  $\T/\T^1\mapsto \R^\times$  into the set of  regular characters  $\Aa_{anti}\rightarrow \R$ such that
$$\overline \phi(\tau_{e^\lambda}):= 
\phi(\widehat {e^{-\lambda}})\textrm{ for all $\lambda
\in \tilde \X^-_*(\T)$}.$$ We denote the inverse map by $\psi\mapsto \underline\psi$. 
\end{lemma}

\begin{proof}  We use \eqref{eq:add} repeatedly to justify the following arguments. The formula given for $\overline\phi$ defines a regular character $\Aa_{anti}\rightarrow \R$. 
Now consider $\psi:\Aa_{anti}\rightarrow \R$ a regular character.
Let $t\in\T$ and denote by $\lambda$ the element in $\tilde \X^*(\T)$ such that $\I t\I=\I \widehat{e^{\lambda}} \I$.  
There are $\lambda_1, \lambda_2\in \tilde \X_*^-(\T)$ such that 
$\lambda=\lambda_1-\lambda_2$ and we set 
$$\underline\psi(t):=\psi(\tau_{e^{\lambda_1}}) \psi(\tau_{e^{\lambda_2}})^{-1}$$ which is well defined because $\psi$ is regular.
Furthermore, one checks that it defines a morphism $\T\rightarrow \R^\times$ which is trivial on $\T^1$.

\end{proof}

Consider  a regular $\R$-character
$ \xi: \Aa_{anti}\rightarrow \R$ and the corresponding morphism $\underline\xi$  which we  see as a map $ \T\rightarrow  \R^\times$  trivial on $\T^1$.
Inflating $\underline\xi$ to a character of the Borel $\B$, we
consider the $\R$-module of the functions $f: \G\rightarrow \R$ satisfying 
$f(bg)= \underline\xi(b) f(g)$ for all $g\in \Gp$, $b\in \B$.  It is endowed with a $\R$-linear action of $\G$ by right translations namely $(g,f)\mapsto f(\, . \,g)$.
We denote by $$\Ind_{\B}^\Gp(\underline\xi)$$  its smooth part and obtain a
smooth $\R$-representation of $\G$.

\begin{lemma}\label{rema:mackey} Let $\Omega$ be a pro-$p$ subgroup of $\K$. The space of $\Omega$-invariant functions
$$(\Ind_{\B}^\Gp(\underline\xi))^{\Omega}$$ is a free $\R$-module of finite rank equal to $\vert  \B\backslash\G/\Omega\vert$. 

\end{lemma}
\begin{proof} The morphism $\underline\xi$ can be seen as a $\k$-representation of $\T$ over the $\k$-vector space $\R$ and therefore, by the classical theory of $\k$-representations of $\Gp$, 
if $\Omega$ is a compact open subgroup of $\Gp$, we have a $\k$-linear isomorphism
$$( \Ind_{\B}^\Gp(\underline\xi))^\Omega \cong \prod_{k\in \B\backslash \Gp/\Omega} \R^{\B\cap k \Omega k^{-1}}$$ given by 
the evaluation of $f \in (\Ind_{\B}^\Gp(\underline\xi))^\Omega$ at all $k$ in a chosen system of representatives of $\B\backslash \Gp/\Omega$. If $\Omega$ is a pro-$p$ subgroup of $\K$ then, by Cartan decomposition,
one can choose $k\in \K$ and  then $\B\cap k \Omega k^{-1}\subseteq \B\cap \K=\B\cap \Iw$.  But $\B\cap k \Omega k^{-1}$  is a pro-$p$ group so  it is contained  in $\B\cap\I$ on which $\underline\xi$ is trivial. We therefore have a $\k$-linear isomorphism 
$$( \Ind_{\B}^\Gp(\underline\xi))^\Omega \cong \prod_{k\in \B\backslash \Gp/\Omega} \R$$  given by 
the evaluation of $f \in (\Ind_{\B}^\Gp(\underline\xi))^\Omega$ at all $k$ in a chosen system of representatives of $\B\backslash \Gp/\Omega$.
This map being obviously $\R$-equivariant, we have proved that 
$( \Ind_{\B}^\Gp(\underline\xi))^\Omega$ is a free $\R$-module of rank $\vert \B\backslash \Gp/\Omega\vert$.

\end{proof}

\begin{prop} We have an isomorphism of $\R$-representations of $\Gp$
$$ \xi\otimes_{\Aa_{anti}}\XX\cong \Ind_{\B}^\Gp(\underline\xi).$$
\label{prop:iso}
\end{prop}
\begin{proof} The proof follows closely the strategy of  \cite[Prop 11, p.80]{SS1} which considers the case of the principal series representation induced by the trivial character with values in $\Z$ in the case of $\mathbf G={\rm GL}_n$.   
In the case of unramified principal series representations of ${\rm GL}_n$ over a ring, and respectively, for more general comparison between compact and parabolic  induction over an algebraically closed field with characteristic $p$,  \cite[4.5]{VigGL2}  
and \cite[Theorem 3.1, Corollary 3.6]{Her} use similar techniques inspired by \cite{SS1}.

Denote by $f_{1}$ the $\I$-invariant function in $\Ind_{\B}^\Gp(\underline\xi)$ with support $\B \I$ and value $1_\R$ at $1_\Gp$. 
Since $\underline\xi$ is trivial on $\B\cap \I$, it is well defined  by the formula  $f_1(b u)= \underline\xi(b)$ for all $b\in \B$ and $u\in \I$.

1/ We consider the  morphism of $\k$-representations of $\G$
$$\Phi: \XX \longrightarrow   \Ind_{\B}^\Gp(\underline\xi)$$
sending the characteristic function ${\rm char}_\I$ of $\I$   onto $f_1$. 
Let $\lambda\in \tilde\X_*^-(\T)$. We compute $f_1\tau_{e^\lambda}$. Decompose $\I \widehat {e^\lambda}\I$ into simple right cosets mod $\I$. By Lemma \ref{negative},  one can find such a decomposition 
 $\I \widehat {e^\lambda}\I=\coprod_{k} \I \widehat {e^\lambda} k$ with $k$ ranging over some finite subset of  $\I^-$.
Now   $f_1\tau_{e^\lambda}$ is $\I$-invariant with support  in $\B \I^-  \widehat {e^\lambda} \I^-=\B \I$. 
To compute its value at $1$,  
one checks  that  for $k\in \I^-$, we have $1\in\B \I^- \widehat {e^\lambda} k$ if and only if 
$\I \widehat {e^\lambda} k= \I \widehat {e^\lambda}$ and therefore
$$f_1\tau_{e^\lambda}(1)= [\widehat {e^\lambda} ^{-1}. f_1](1)=\underline\xi(\widehat {e^{-\lambda}})=\xi(\tau_{ {e^\lambda}}).$$
We have proved that $\Phi(\tau_{e^\lambda})=\xi(\tau_{ {e^\lambda}})\Phi({\rm char}_\I)$. It proves that  $\Phi$ induces
a
 morphism of $\R$-representations   of $\G$ $$\Phi':  \xi\otimes_{\Aa_{anti}}\XX \longrightarrow   \Ind_{\B}^\Gp(\underline\xi).$$

2/  We show that  $f_1$ generates $\Ind_{\B}^\Gp(\underline\xi)$ as a $\R$-representation of $\Gp$.
Let $f\in  \Ind_{\B}^\Gp(\underline\xi)$. Its support is open and compact in $\B\backslash \G$  and by Lemma \ref{lemmaSS}, we can suppose (after restricting and translating) that $f$ has support in $\B \U^-$. The restriction  $f\vert_{\U^-}$ is locally constant and we can suppose (after restricting the support more) that $f\vert _{\U^-}$ is constant  on some compact open subset $\EuScript C$. 
By Lemma  \ref{lemmaSS'}, this set $\EuScript C$ is  the finite union  of subsets of the form $ t^{-n }\I^- t^n u$ for $n$ large enough and $u\in \U^-$, where $t$ is defined in \S\ref{strongl}. Restricting again (and translating), one can suppose that $f\vert_{\U^-}$ has support $ t^{-n }\I^- t^n$ and is  constant with value $r\in \R$ on this subset.
Now for all $(b,u)\in \B\times \I$,  write $u= u^+ u_0 u^- \in\I^+\I^0\I^-$ and recall that $\xi(u^+ u^0)=1$.
We have $ (t^n f)(bu)= f(b u^+ u_0 t^n  t^{-n} u ^- t^n)=\underline\xi(b u^+ u_0  t^n) r=\underline\xi(b)\underline\xi( t^n) r$. Therefore,
$f=\underline\xi( t^n) r \: \,(t^{-n}.f_1)$ lies in the sub-$\R$-representation generated by $f_1$.  This proves that $\Phi'$ is surjective.

3/ To prove that $\Phi'$ is injective we follow the strategy of  \cite[pp.80 \& 81]{SS1}. 
For $n\in\N$, denote by $\mathbf Y_n$ the subspace of $\XX$ of the functions with support in  $\I  t^n \K$.

\begin{fact}  Consider an element in $\xi\otimes_{\Aa_{anti}}\XX$. There is $n\in \N$ such that it  can be written as a sum of elements of the form $r\otimes f$ where $r\in \R$  and $f\in \mathbf Y_n$.\label{f1}
\end{fact}

\begin{fact} For $k\in \K$ and $n\in \N$, we have
$\B \I t^n  k \cap \B\I t^n\neq \emptyset$ if and only if $ \I t^n  k = \I  t^n$.
\label{f2}

\end{fact}

The facts together prove the injectivity of $\Phi'$.
\begin{proof}[Proof of the facts] The proof of Fact  \ref{f2} in the case of $\mathbf G={\rm GL}_n$  given in \cite[p. 81]{SS1} and \cite[Proposition 7, p.77]{SS1} is the same in the general case of a split group.  For Fact \ref{f1}, 
we first notice  that the statement of  \cite[Lemma 12, p.80]{SS1}  holds in the case of a general split group since $(\G, \I', N_G(\T))$ is a generalized Tits system. 
Therefore, for any $g\in \G$, there is $y\in \T^{++}$ such that $\I y \I g\subseteq \I \T^{++}\K$.  The element $1\otimes \char_{\I g}$ can be written $\xi(\tau_y)^{-1}\otimes \char_{\I y \I g}$. Therefore, an element in  $\xi\otimes_{\Aa_{anti}}\XX$  can be written as a sum of elements of the form $r'\otimes f'$ where $r'\in \R$  and $f'$ has support in $\I \T^{++}\K$.
Now let $y'\in \T^{++}$ and $k\in \K$. One can find $n\in \N$ large enough such that $t^{n} {y'}^{-1}= y''\in \T^{++}$. 
Hence the element $r\otimes \char_{\I y' k}$ can be written $r\xi(\tau_{y''})^{-1}\otimes \char_{\I y'' \I y' k}$  and  by \eqref{eq:add} 
we have $\I y'' \I y' k\subseteq  \I y'' \I y' \I k= \I t^n \I k\subseteq \I t^n \K$.

\end{proof}

\end{proof}

\begin{prop} \label{prop:invariants}
 As a right  $\R\otimes_\k\Hh$-module,   $(\Ind_{\B}^\Gp(\underline\xi))^\I$ is isomorphic to $\xi\otimes _{\Aa_{anti}} \Hh$.

\end{prop}

\begin{proof} 
By Lemma \ref{rema:mackey}, 
the $\R$-module $( \Ind_{\B}^\Gp(\underline\xi))^\I$  is free
of rank $\vert \B\backslash\Gp/\I \vert=\vert \Wf\vert$. More precisely,    for any $w\in \Wf$, fix  lifts $\tilde w\in \tilde \Wf$  and $\hat {\tilde w}\in N_\G(\T)$ for $w$,  and denote by $f_w$ the  function in $( \Ind_{\B}^\Gp(\underline\xi))^\I$ with support $\B \hat{\tilde w} \I$ and value $1_\R$ at 
 $\hat{\tilde w}$. The family $(f_w)_{w\in\Wf}$ is a basis for the free $\R$-module  $( \Ind_{\B}^\Gp(\underline\xi))^\I$ (see for example \cite[5.5.1]{Compa}  for more detail). 
By  \cite[Proposition 5.16]{Compa},  the 
composition 
\begin{equation}\label{compo}\xi\otimes_{\Aa_{anti}}\Hh\longrightarrow (\xi\otimes_{\Aa_{anti}}\XX)^\I \overset{\Phi'}\longrightarrow  ( \Ind_{\B}^\Gp(\underline\xi))^\I\end{equation} is  a surjective morphism of $\R\otimes_\k\Hh$-modules since the image of $1_\R\otimes \tau_{\tilde w}$ is equal to $f_w$ for all $w\in \Wf$.  From  \eqref{mult} and since $\xi$ is regular,
we deduce that $\xi\otimes_{\Aa_{anti}}\Hh$ is generated as an $\R$-module  by the set of all $1_\R\otimes \tau_{\tilde w}$ for   $w\in \Wf$. This is enough to prove that  \eqref{compo} is injective.

\end{proof}

By Propositions \ref{prop:iso} and \ref{prop:invariants}, there are  natural isomorphisms of $\R$-representations of $\G$  \begin{equation} 
\xi\otimes _{\Aa_{anti}}\XX\cong 
(\Ind_{\B}^\Gp(\underline\xi))^\I \otimes_{\Hh}\XX\cong  \Ind_{\B}^\Gp(\underline\xi)\label{theiso}.\end{equation} 
For any facet $F$ of $C$ containing $x_0$ in its closure, they induce
 morphisms of $\R$-representations of $\Pp_F^\dagger$: \begin{equation} \xi\otimes _{\Aa_{anti}}\XX^{\I_F}\cong 
(\Ind_{\B}^\Gp(\underline\xi))^\I \otimes_{\Hh}\XX^{\I_F}\longrightarrow (\Ind_{\B}^\Gp(\underline\xi))^{\I_F}.\label{theisoF}\end{equation}

 We  identify $\k[\T^0/\T^1]$ with its image in  $\Aa_{anti}$ via $t\mapsto \tau_{t^{-1}}$.
The $\Aa_{anti}$-module $\R$ therefore inherits a structure of $\k[\T^0/\T^1]$-module and this structure is given by the restriction of $\underline \xi$ to $\T^0/\T^1$.
Below, we  also consider $\underline \xi$  (or rather its restriction to $\T^0/\T^1$) as a character of $\Iw$ trivial on $\I$.

\begin{lemma}

Let $F$ be a facet of $C$ containing the hyperspecial vertex $x_0$ in its closure.
There is a natural isomorphism of $\R[[\K]]$-modules 
$$\R\otimes _{\k[\T^0/\T^1]}\XX_{x_0}\cong \Ind_{\Iw}^{{\K}}(\underline\xi).$$
It induces an isomorphism of    $\R[[\mathbf{G}_F^\circ(\Oo)]]$-modules
$$\R\otimes _{\k[\T^0/\T^1]}\XX_{x_0}^{\I_F}\cong (\Ind_{\Iw}^{{\K}}(\underline\xi))^{\I_F}.$$
\label{lemma:finite}
\end{lemma}

\begin{proof} 
The first abstract isomorphism  is clear because, as representations of $\K$, we have  $\XX_{x_0}\cong \Ind_{\Iw}^\K \k[\T^0/\T^1]$  and the tensor product commutes with compact induction. We describe this isomorphism explicitly in order to deduce the second one.
 Denote by  $\varphi$ the function in  $\Ind_{\Iw}^{{\K}}(\underline\xi)$ with support $\Iw$ and value $1_\R$ at $1_{\K}$. It is $\I$-invariant.

The following well defined map  realizes the first isomorphism of $\R[[\K]]$-modules:
$$\begin{array}{ccl}\R\otimes _{\k[\T^0/\T^1]}\XX_{x_0}&\longrightarrow &\Ind_{\Iw}^{{\K}}(\underline\xi).\cr  r\otimes {\char}_{\I}&\longmapsto & \ r\varphi.\cr \end{array}$$ 

Consider the $\k[\T^0/\T^1]$-module $\XX_{x_0}^{\I_F}$.
It is free  with basis 
 the set of all $\char_{\I x \I_F}$ for $x$ ranging over a system of representatives of $\Iw\backslash \K/\I_F$.  This can be seen by noticing that $\I' x\I_F$ is the disjoint union of all $\I t x \I_F$ for   $t\in\T^0/\T^1$.
 In particular, $\XX_{x_0}^{\I_F}$ is projective and therefore injective over the Frobenius algebra $\k[\T^0/\T^1]$: it is a direct summand of $\XX_{x_0}$  and  we have an injective morphism
of   $\R[[\mathbf{G}_F^\circ(\Oo)]]$-modules
\begin{equation}\label{mapF}\R\otimes _{\k[\T^0/\T^1]}\XX_{x_0}^{\I_F}\hookrightarrow (\R\otimes _{\k[\T^0/\T^1]}\XX_{x_0})^{\I_F} \cong (\Ind_{\Iw}^{{\K}}(\underline\xi))^{\I_F}.\end{equation}
For $x\in \K$, the $\I_F$-invariant function in $(\Ind_{\Iw}^{{\K}}(\underline\xi))^{\I_F}$ with support $\I' x\I_F$ and value $r\in \R$ at $x$ is the image by \eqref{mapF}
of $r\otimes \char_{\I x \I_F}$. Therefore \eqref{mapF} is surjective.


 \end{proof}

\begin{prop}\label{prop:free1}

If  $F$ is a facet of $C$ containing $x_0$ in its closure,  then \eqref{theisoF} is a chain of isomorphisms \begin{equation} \xi\otimes _{\Aa_{anti}}\XX^{\I_F}\cong 
(\Ind_{\B}^\Gp(\underline\xi))^\I \otimes_{\Hh}\XX^{\I_F}\cong (\Ind_{\B}^\Gp(\underline\xi))^{\I_F}.\end{equation}
 of $\R$-representations of  $\Pp_F^\dagger$.
In particular, the $\R$-module $ \xi\otimes_{\Aa_{anti}}  \mathbf{X}^{\I_F}$ is free.

\end{prop}
\begin{proof}  
There is a well defined morphism of $\R$-representations  of $\K$ \begin{equation}\label{isopreli}\Ind_{\Iw}^{{\K}}(\underline\xi)\longrightarrow  (\Ind_{\B}^\Gp(\underline\xi))^{\K_1}\end{equation} defined by sending the function $\varphi$  onto the  function $f_1\in (\Ind_{\B}^\Gp(\underline\xi))^{\I}$ (notations of the  proof of Lemma \ref{lemma:finite} and Proposition \ref{prop:iso}).
 It is injective since for $k\in \K$, the equality $\B \I k\cap \B \I\neq \emptyset$ implies $k\in \Iw$. 
By Iwasawa decomposition and since $\K_1$ is normal in $\K$,  the $\R[[\K]]$-module $(\Ind_{\B}^\Gp(\underline\xi))^{\K_1}$ is  generated  by the $\K_1$-invariant function with support in $\B \K_1=\B \I$  and value $1_\R$ at $1_\G$. This function is in fact equal to $f_1$ because $\underline\xi$ is trivial on $\I^+$.  Therefore \eqref{isopreli} is an isomorphism.

\medskip

We want to show that the natural morphism
 of $\R$-representations of $\Pp_F^\dagger$
\begin{equation}
(\Ind_{\B}^\Gp(\underline\xi))^\I \otimes_{\Hh}\XX^{\I_F}\longrightarrow(\Ind_{\B}^\Gp(\underline\xi))^{\I_F}\end{equation} is bijective.  By \eqref{transf}, it is enough to show that
the natural morphism of 
$\R[[\mathbf{G}_F^\circ(\Oo)]]$-modules
\begin{equation}
(\Ind_{\B}^\Gp(\underline\xi))^\I \otimes_{\H_F}\XX_F\longrightarrow(\Ind_{\B}^\Gp(\underline\xi))^{\I_F}\end{equation} is bijective. 
Since $x_0$ is in the closure of $F$,
passing to $\I$-invariant vectors in 
\eqref{isopreli} yields an isomorphism of right $\R\otimes_\k\H_{x_0}$-modules and  therefore of
right $\R\otimes_\k\H_F$-modules. Likewise, passing to $\I_F$-invariant vectors yields an isomophism of $\R[[\mathbf{G}_F^\circ(\Oo)]]$-modules.
Therefore we want to show that the natural morphism of $\R[[\mathbf{G}_F^\circ(\Oo)]]$-modules
\begin{equation}\label{s1}
(\Ind_{\Iw}^{{\K}}(\underline\xi)) ^{\I}\otimes_{\H_F}\XX_F\longrightarrow(\Ind_{\Iw}^{{\K}}(\underline\xi))^{\I_F}\end{equation} is bijective. 
Now by Lemma \ref{lemma:finite} and using \eqref{transfF}, we check that \eqref{s1}  can be decomposed into the following    chain of isomorphisms
$$(\Ind_{\Iw}^{{\K}}(\underline\xi))^\I \otimes_{\H_F}\XX_F \simeq  \R\otimes _{\k[\T^0/\T^1]}\H_{x_0}\otimes  _{\H_F}\XX_F\cong \R\otimes _{\k[\T^0/\T^1]}\XX_{x_0}^{\I_F}\cong (\Ind_{\Iw}^{{\K}}(\underline\xi))^{\I_F}.$$  

\end{proof}

\begin{prop} The  $\R$-module $\Ind_{\B}^\Gp(\underline\xi)\simeq\xi\otimes_{\Aa_{anti}} \mathbf X$ is free.
\label{prop:free2}

\end{prop}

\begin{proof} As an $\R$-module, $\Ind_{\B}^\Gp(\underline \xi)$ is the inductive limit of the family $((\Ind_{\B}^\Gp(\underline \xi))^{\K_m})_{m\geq 0}$ where we set $\K_0=\I$. We prove the proposition by first invoking  Lemma \ref{rema:mackey} which ensures  that $\Ind_{\B}^\Gp(\underline \xi)^{\K_0}$ is a free (finitely generated) $\R$-module, and then by proving that for all $m\geq 0$, the quotient $(\Ind_{\B}^\Gp(\underline \xi))^{\K_{m+1}}/(\Ind_{\B}^\Gp(\underline \xi))^{\K_m}$ is a free (finitely generated) $\R$-module. 
For this, let $m\geq 0$. 
For $g\in\G$,  denote by  $\Ind_{\B}^{\B g \K_m}(\underline \xi)$ the subspace of the functions in  $\Ind_{\B}^{\G}(\underline \xi)$ with support in $\B g \K_m$ and decompose the latter  into a  finite disjoint union  $\B g\K_m=\coprod_{i=1}^s \B g k_i \K_{m+1}$.
By Lemma \ref{rema:mackey}, the map 
\begin{eqnarray}  \label{m} (\Ind_{\B}^{\B g \K_m}(\underline \xi))^{\K_{m+1}}\longrightarrow & \R^s\cr \: f\longmapsto& (f(gk_i))_{\1\leq i\leq s}\end{eqnarray}
is a $\R$-linear isomorphism. A function  $f\in  (\Ind_{\B}^{\B g \K_m}(\underline \xi))^{\K_{m+1}}$ is $\K_m$-invariant  if and only if its image by \eqref{m} lies in the  submodule $D$ of $\R^s$ generated by $(1)_{\1\leq i\leq s}$. Since $\R^s/D$ is a free $\R$-module and    $(\Ind_{\B}^\Gp(\underline \xi))^{\K_{m+1}}/(\Ind_{\B}^\Gp(\underline \xi))^{\K_m}$ is isomorphic to the   direct sum of
all  $(\Ind_{\B}^{\B g \K_m}(\underline \xi))^{\K_{m+1}}/(\Ind_{\B}^{\B g \K_m}(\underline \xi))^{\K_m}$ for $g$ in the (finite) set  $\B\backslash \G/ \K_m$, we obtain the expected result.

\end{proof}

\section{\label{sec:reso}Resolutions for principal series representations of ${\rm GL}_n(\mathfrak F)$ in arbitrary characteristic}

Let $\chi: \T\rightarrow \k^\times$ a morphism of groups which we suppose to be trivial on $\T^1$.  We are interested in the principal series $\k$-representation of $\Gp$
$$\mathbf V=\Ind_\B^\G(\chi).$$
and the associated coefficient system $\cV$ defined in  Section \ref{therez}.
As in Section \ref{sec:over-a-ring},  we consider the sub-$\k$-algebra $\Aa_{anti}$ of  $\Hh$  and we  attach to $\chi$ the regular $\k$-character
$\overline\chi: \Aa_{anti}\rightarrow k$  as in Lemma \ref{lemma:xi}. 
Define $\R$ to be the localization of $\Aa_{anti}$ at the kernel of $\overline\chi$ and $\xi: \Aa_{anti}\rightarrow \R$ to be the natural morphism of localization. 
 It is a regular character of $\Aa_{anti}$. There is  a $\k$-character $\overline{\overline\chi}:\R\rightarrow \k$ satisfying $\overline{\overline\chi}\circ \xi=\overline \chi.$
Since $\R$ is a flat $\Aa_{anti}$-module,  
tensoring  the complex of  left $\Aa_{anti}$-modules \eqref{chain-complex} by $\R$ yields an exact sequence of $\R$-representations of $\Gp$:
\begin{equation}\label{chain-complex-local}
    0 \longrightarrow \xi\otimes _{\Aa_{anti}} C_c^{or} (\mathscr{X}_{(d)}, \cX) \xrightarrow{\;\;} \ldots \xrightarrow{\;\;}  \xi\otimes _{\Aa_{anti}}C_c^{or} (\mathscr{X}_{(0)}, \cX) \xrightarrow{\;\;}  \xi\otimes _{\Aa_{anti}}\mathbf{X} \longrightarrow 0
\end{equation}

Suppose that $\mathbf G= {\rm GL}_n$ for $n\geq 1$. Then for any $i\in\{0, \ldots,d\}$, we  can choose
 the facets in  $\Ff_i$  to contain $x_0$ in their closure. 
Therefore, by  Propositions  \ref{prop:sum}, \ref{prop:free1} and \ref{prop:free2}, all the terms of the exact complex  \eqref{chain-complex-local} are free $\R$-modules. The complex splits as a complex of $\R$-modules and it remains exact after tensoring by the $\k$-character $\overline{\overline\chi}$ of $\R$. But $\overline{\overline\chi}\otimes_\R\xi$ is isomorphic to the space $\k$ endowed with the structure of $\Aa_{anti}$-module given by $\overline\chi:\Aa_{anti}\rightarrow \k$.
By Proposition \ref{prop:iso}, this gives a $\Gp$-equivariant resolution of $\Ind_\B^\G(\chi)$:

\begin{equation}\label{chain-complex-local2}
    0 \longrightarrow \overline\chi\otimes _{\Aa_{anti}} C_c^{or} (\mathscr{X}_{(d)}, \cX) \xrightarrow{\;\;} \ldots \xrightarrow{\;\;}  \overline\chi\otimes _{\Aa_{anti}}C_c^{or} (\mathscr{X}_{(0)}, \cX) \xrightarrow{\;\;}  \Ind_\B^\G(\chi)\longrightarrow 0
\end{equation}

 This complex is isomorphic to the augmented  complex associated to the coefficient system on $\mathscr X$ denoted by $  \overline \chi  \otimes _{\Aa_{anti}}\:\cX $ and defined  by 
 $F\longmapsto  \overline \chi \otimes _{\Aa_{anti}}\XX^{\I_F}$  for any facet in $\mathscr X$. By  Proposition \ref{prop:free1},
  $\overline \chi  \otimes _{\Aa_{anti}}\:\cX $
 is isomorphic to $\cV$.  Therefore, the  exact complex  \eqref{chain-complex-local2}  is isomorphic to the  complex \eqref{chain-complexintroV} 
 and we have proved Theorem \ref{theo}.
 Note that by Proposition \ref{prop:sum}, the exact resolution  \eqref{chain-complexintroV} is of the form: \\
  \begin{equation}\label{chain-complex-local3}\\
    0 \longrightarrow  \bigoplus_{F\in \Ff_{d}}\ind_{\Pp_F^\dagger}^\Gp(  (\Ind_\B^\G  \chi) ^{\I_F}\otimes \epsilon_F) \xrightarrow{\;\;} \ldots \xrightarrow{\;\;}  \bigoplus_{F\in \Ff_0}\ind_{\Pp_F^\dagger}^\Gp(  (\Ind_\B^\G  (\chi)) ^{\I_F}\otimes \epsilon_F)\xrightarrow{\;\;}  \Ind_\B^\G(\chi)\longrightarrow 0.
\end{equation}
Here since $\mathbf G= {\rm GL}_n$, the semisimple rank is  $d=n-1$.

\section{A remark about  the Schneider-Vignéras functor \label{sec:sv}}
Assume that $\G= {\rm GL}_n(\mathbb Q_p)$ with $n\geq 2$, and denote by $Z$ its center.   We set $\B_0:=\B\cap \K$. It is a subgroup of $\I'$.

\begin{lemma} \label{lemma:B0}
Let $F$ be a facet of  the standard apartment $\Ap$ containing $x_0$ in its closure. We have 
$ \Pp_F^\dagger \cap \B\subset \B_0 Z.$
\end{lemma}

\begin{proof}   Any vertex in the closure of $\F$ is of the form $\widehat{e^{\lambda}} x_0$ for some
$\lambda\in \X_*(\T)$ and this vertex coincides with $x_0$ if and only if $\widehat{e^{\lambda}}$ is in the center $Z$ that is to say if $\lambda\in  \X_*(Z)$. Let $b\in \Pp_F^\dagger \cap \B$. There is ${\lambda_1} \in \X_*(\T)$ such that $\widehat{e^{\lambda_1}}x_0\in F$ and $b x_0= \widehat{e^{\lambda_1}}x_0$. Therefore $b\in  \widehat{e^{\lambda_1}}\K Z\cap \B=\widehat{e^{\lambda_1}}\B_0 Z$.  Write $b=  \widehat{e^{\lambda_1}}u z$ with $u\in \B_0$ and $z\in Z$. Inductively, we construct a sequence $(\lambda_m)_{m\geq 1}$ in $\X_*(\T)$ such that $\widehat{e^{\lambda_{m}}}x_0\in F$  and
$b \widehat{e^{\lambda_{m}}} x_0= \widehat{e^{\lambda_{m+1}}}x_0$. It implies  $\widehat{e^{\lambda_1-\lambda_{m+1}}}u   \widehat{e^{\lambda_{m}}} \in \K Z$. Looking at the diagonal of this element, we find
$\lambda_{m+1}=\lambda_1 +\lambda_m \mod \X_*(Z)$ and therefore $\lambda_m= m \lambda_1 \mod \X_*(Z)$ for any $m\geq 1$. If  $\lambda_1\not\in  \X_*(Z)$, then the family of all  $\widehat{e^{\lambda_{m}}} x_0$ is infinite: we obtain a contradiction. Therefore $b\in \B_0 Z$.

\end{proof}

We can identify $\Wf$ with a subgroup of $\G$ and $\Wf$ yields a system of representatives of the double cosets
$\I \backslash \Gp /\B$. For any $i\in\{0, ..., n-1\}$, choose the facets in  $\Ff_i$  to contain $x_0$ in their closure.   For $F\in \Ff_i$, we  choose a system of representatives of 
$\Pp_F^\dagger\backslash \Gp /\B$ in $\Wf$. For $w\in \Wf$,   we can apply Lemma \ref{lemma:B0}.
 to the facet $w^{-1}F$  of $\Ap$.
\medskip

 Let $\chi: \T\rightarrow \k^\times$ a morphism of groups which is trivial of $\T^1$.
Restricting  \eqref{chain-complex-local3} to a complex of $\k$-representations of $\B$, we obtain an exact complex:

 \begin{equation*}\label{chain-complex-local-B}
 \begin{split}
    0 \rightarrow  \!\!\! \!\bigoplus _{F\in \Ff_{n-1}, \atop \;w\in \Pp_F^\dagger \backslash \Gp /\B}\! \! \!\! \!\!\ind_{ w^{-1}\Pp_F^\dagger w\cap \B}^\B (w\star ( (&\Ind_\B^\G  (\chi)) ^{\I_F}\otimes \epsilon_F))\xrightarrow{\;\;} \ldots \cr \ldots& \xrightarrow{\;\;}  \bigoplus_{F\in \Ff_0, \atop \;w\in \Pp_F^\dagger \backslash \Gp /\B}  \!\! \! \!\!\! \ind_{ w^{-1}\Pp_F^\dagger w\cap \B}^\B( w\star ( (\Ind_\B^\G  (\chi)) ^{\I_F}\otimes \epsilon_F))\xrightarrow{\;\;}  \Ind_\B^\G(\chi)\vert_\B\rightarrow 0\cr\end{split}
\end{equation*} where  $w\star ( (\Ind_\B^\G  (\chi)) ^{\I_F}\otimes \epsilon_F)$ denotes the space  $(\Ind_\B^\G  (\chi)) ^{\I_F}\otimes \epsilon_F$ with the group  $w^{-1}\Pp_F^\dagger w\cap \B$
 acting through the homomorphism 
$ w^{-1}\Pp_F^\dagger w\cap \B \xrightarrow{w \,.\, w^{-1}} \Pp_F^\dagger$. 
Therefore, applying Lemma \ref{lemma:B0},   there exist smooth $\k$-representations $ V_0$, ..., $V_{n-1}$ of $\B_0Z$ and an exact resolution of the restriction to $\B$ of $ \Ind_\B^\G(\chi)$  of the form:

 \begin{equation}\label{resoSV}
 \mathcal I_\bullet:   0 \longrightarrow  \ind_{\B_0Z}^\B (V_{n-1}) \xrightarrow{\;\partial_{n-1}\;} \ldots \xrightarrow{\;\;\;\;\;\;}  \ind_{\B_0Z}^\B ( V_0) \xrightarrow{\;\partial_{0}\;}  \Ind_\B^\G(\chi)\vert_\B\longrightarrow 0
\end{equation}

\medskip
From now on, $\k$ has characteristic $p$.
As noted by Z\`abr\`adi in \cite[\S4]{Z}, the argument of \cite[Lemma 11.8]{SV}  generalizes to the case of ${\rm GL}_n(\mathbb Q_p)$. Therefore,  we can compute the image of
$\Ind_\B^\G(\chi)\vert_\B$
by the universal $\delta$-functor $ V \mapsto D^i(V)$, $i\geq 0$,   defined in  \cite{SV} using the cohomology  of the complex $D(\mathcal I_\bullet)$: for $i\geq 0$
$$D^i(\Ind_\B^\G(\chi)\vert_\B)= h^i(D(\ind_{\B_0Z}^\B ( V_0))\xrightarrow{\;D(\partial_{0})\;}D( \ind_{\B_0Z}^\B ( V_1))\rightarrow ... \xrightarrow{\;D(\partial_{n-1})\;}  D(\ind_{\B_0Z}^\B  (V_{n-1}))\rightarrow 0\rightarrow 0...)$$
By \cite[Remark 2.4, i]{SV}, the map $D(\partial_{n-1})$ is surjective. Therefore, we have proved that 

$$D^i(\Ind_\B^\G(\chi)\vert_\B)=0\textrm{ for all }i\geq n-1.$$

\end{document}